\documentclass[a4paper]{amsart}
\usepackage{amssymb}
\usepackage[dvips,all]{xy}
\usepackage{enumitem}
\usepackage{pstricks}

\theoremstyle{plain}
\newtheorem{thm}{Theorem}[section]
\newtheorem{cor}[thm]{Corollary}
\newtheorem{lem}[thm]{Lemma}
\newtheorem{prop}[thm]{Proposition}
\newtheorem{conj}[thm]{Conjecture}
\theoremstyle{definition}
\newtheorem{defn}[thm]{Definition}
\newtheorem{remark}[thm]{Remark}
\newtheorem{ex}[thm]{Example}

\newcommand\calA{\mathcal{A}}            
\newcommand\calB{\mathcal{B}}            \newcommand\calO{\mathcal{O}}

\newcommand\calE{\mathcal{E}}

\newcommand\calM{\mathcal{M}}            \newcommand\calZ{\mathcal{Z}}

\newcommand\cat[1]{\ensuremath{\mathbf{#1}}}
\newcommand\op{\mathrm{op}}
\newcommand\id{\mathrm{id}}
\newcommand\sq{\mathrel{\square}}
\newcommand\pt{\mathrm{point}}
\newcommand\modu{/\mathord{\simeq}}
\newcommand\rel{\mathrel{\mathrm{rel}}}
\DeclareMathOperator{\ori}{or}
\DeclareMathOperator{\ob}{ob}

\newcommand{\BIGOP}[1]{\mathop{\mathchoice%
{\raise-0.22em\hbox{\huge $#1$}}%
{\raise-0.05em\hbox{\Large $#1$}}{\hbox{\large $#1$}}{#1}}}
\newcommand{\bigtimes}{\BIGOP{\times}}
\newcommand{\BIGboxplus}{\mathop{\mathchoice%
{\raise-0.35em\hbox{\huge $\boxplus$}}%
{\raise-0.15em\hbox{\Large $\boxplus$}}{\hbox{\large $\boxplus$}}{\boxplus}}}

\begin{document}

\title[Higher order track categories]{Higher order track categories and the algebra of higher order cohomology operations}

\author{Hans-Joachim Baues}

\begin{abstract}
	We describe a conjecture on the algebra of higher cohomology operations which leads to the computations of the differentials
	in the Adams spectral sequence. For this we introduce the notion of an n-th order track category which is suitable to study
	higher order Toda brackets and the differentials in spectral sequences. We describe various examples of higher order track 
	categories which are topological, in particular the track category of higher cohomology operations.
	Also differential algebras give rise to higher order track categories.
\end{abstract}


\maketitle

\section{Deformation categories}\label{DefCat}
We first recall properties of cylinders of spaces. Let $I=[0,1]$ be the unit interval. For a subspace $A\subset B$ we define the 
\emph{relative cylinder} $I_A B$ by the pushout diagram in \cat{Top}
\begin{equation*}
	\xymatrix{
	I\times B\ar[r]^{\overline{p}}&	I_A B\\
	I\times A\ar[u]^{I\times i}\ar[r]_p & A\ar[u]_{\overline{i}}
	}
\end{equation*}
where $i:A\subset B$ is the inclusion and $p$ is the projection of the product space $I\times A$. 
We have \emph{inclusions}
\begin{equation*}
	i^+,i^-:B\to I_AB
\end{equation*}
by $i^+(x)=\overline{p}(1,x),\ i^-(x)=\overline{p}(0,x),\ x\in B$. We call the union
\begin{equation*}
	\partial(I_AB)=B^-\cup B^+\text{ with } B^-=i^-B,\ B^+=i^+B
\end{equation*}
the \emph{boundary} of $I_A B$ with $B^-\cap B^+=\overline{i}A$. 
We also have a \emph{projection} $q:I_AB\to B$ with $q\overline{p}=p$ and 
\begin{equation*}
	qi^+=1 \text{ and } qi^-=1
\end{equation*}
A map $H:I_AB\to C$ in \cat{Top} is termed a \emph{homotopy $H:f\simeq g \rel A$} where $f=Hi^-$ and $g=Hi^+$.
The map 
\begin{equation*}
	\varepsilon(f)=fq:I_AB\to B\to C
\end{equation*}
is the \emph{constant homotopy} $\varepsilon(f):f\simeq f\rel A$.
Moreover the map $I\to I$ which carries $t$ to $1-t$ induces for $H:f\simeq g\rel A$ the \emph{opposite homotopy}
\begin{equation*}
	H^{\op}:g\simeq f\rel A
\end{equation*}
By \emph{pasting homotopies} $H:f\simeq g\rel A$ and $G:g\simeq h\rel A$ one gets
\begin{equation*}
	G\mathrel{\square} H:f\simeq h\rel A
\end{equation*}

Let \cat{Top} be a convenient category of topological spaces and let $\cat{Top}^*$ be the category of pointed spaces $(X,*)$.
For pointed spaces $X, Y$ we have the smash product $X\wedge Y=X\times Y/X\vee Y$ where
$X\vee Y=X\times* \cup *\times Y$. If $A$ is a non-pointed space then $A^+=A\overset{.}{\cup}*$ is the disjoint union
with a base point and we have
\begin{equation*}
	A^+\wedge X = A\times X / A\times *
\end{equation*}
We call a map $f:A^+\wedge X\to Y$ in $\cat{Top}^*$ an \emph{$A$-map $X\to Y$}. We can identify the $A$-map $f$ with a 
map $f:A\to (Y,*)^{(X,*)}$ where $(Y,*)^{(X,*)}$ is the function space of pointed maps $X\to Y$.

\begin{defn}\label{CatDef}
	Let $\cat{C}(A)$ be the category of $A$-maps. Objects are pointed spaces $X, Y$ and morphisms are $A$-maps. Composition $gf$ of $A$-maps
	is determined by the composition
	\begin{equation*}
		A\times X\xrightarrow{\Delta\times X} A\times A\times X\xrightarrow{A\times f} A\times Y \xrightarrow{g} Z
	\end{equation*}
	where $\Delta$ is the diagonal of $A$. The identity of $X$ in $\cat{C}(A)$ is given by the projection $A\times X\to X$.
\end{defn}

A map $i:A\to B$ in \cat{Top} induces a functor
\begin{equation*}
	i^*:\cat{C}(B)\to \cat{C}(A)
\end{equation*}
which is the identity on objects and carries a $B$-map $f$ to the $A$-map $i^*f$ given by the composite
$f(i\times X): A\times X\to B\times X\to Y$.
The category $\cat{C}(A)$ has a \emph{zero object $*$} (that is, an initial and final object) given by the basepoint $*$.

We now consider the category of $I_AB$-maps for a relative cylinder $I_AB$.

\begin{prop}
	Let $\cat{C}_1=\cat{C}(I_A B)$ and $\cat{C}_0=\cat{C}(B)$. Then $\cat{C}_0, \cat{C}_1$ are categories with a zero
	object together with functors
	\begin{gather*}
		\cat{C}_0\xrightarrow{\varepsilon}\cat{C}_1\overset{\partial^+}{\underset{\partial^-}{\rightrightarrows}}\cat{C}_0,\\
		\op:\cat{C}_1\to \cat{C}_1,\\
		\square:\cat{C}_1\times_{\partial}\cat{C}_1\to \cat{C}_1
	\end{gather*}
	which are the identity on objects. Here $\cat{C}_1\times_{\partial}\cat{C}_1$ is the category of pairs $(G,H)$ of maps
	$X\to Y$ in $\cat{C}_1$ satisfying $\partial^+H=\partial^-G$. The functors satisfy the equations (where $\id$ denotes the
	identity functor)
	\begin{gather*}
		\partial^+\varepsilon =\partial^-\varepsilon=\id,\\
		\partial^+\op=\partial^-,\quad \op\circ \varepsilon=\varepsilon,\\
		\op\circ \op = \id,\\
		\op(G\mathrel{\square}H)=H^{\op}\mathrel{\square}G^{\op}
\end{gather*}
\end{prop}
We call such a pair $(\cat{C}_0,\cat{C}_1)$ a \emph{deformation category}.

\begin{proof}
	With the structure maps of the relative cylinder we define $\varepsilon=q^*$ for $q:I_AB\to B,\ \partial^+=(i^+)^*,\ \partial^-=(i^-)^*$. 
	Moreover $\op$ and $\square$ are given by the opposite homotopy and by pasting respectively.
\end{proof}

In a deformation category $(\cat{C}_0,\cat{C}_1)$ we write $H:f\simeq g$ if $H$ is a morphism in $\cat{C}_1$ with
$\partial^-H=f$ and $\partial^+H=g$.

Deformation categories form a category. Morphisms are pairs of functors 
$F_0:\cat{C}_0\to\cat{C}_0',\ F_1:\cat{C}_1\to \cat{C}_1'$ 
which are the identity on objects and which are compatible with $\varepsilon, \op, \partial^+, \partial^-$
and $\square$.

\section{Track categories}\label{TraCat}

A track category by definition is a groupoid enriched category with a strict zero object. In this section we compare
track categories and deformation categories. We consider examples of abelian track categories and triple Toda brackets.

\begin{lem}
	A track category is a deformation category.
\end{lem}

\begin{proof}
	Let \cat{C} be a track category. Then we define the category $\cat{C}_0$ by the category of $1$-cells $f$ in \cat{C}.
	The category $\cat{C}_1$ has the morphisms $X\to Y$ which are triples $(f,g,H)$ where $H:f\Rightarrow g$ is a $2$-cell
	(also termed \emph{track}). Composition $(f,g,H)(f',g',H')$ is defined by $(ff',gg',H*H')$ where $H*H'$ is the
	horizontal composition of tracks defined by 
	\begin{equation*}
		H*H'=(g')^*H \mathrel{\square} f_*H'=g_*H'\mathrel{\square}(f')^* H
	\end{equation*}
	Here $\square$ is the composition of $2$-cells (termed \emph{pasting} of tracks). We define $\varepsilon(f)=(f,f,0_f)$
	where $0_f$ is the identity $2$-cell of $f$.  Moreover $\partial^-(f,g,H)=f,\ \partial^+(f,g,H)\allowbreak =g$ and
	$\op(f,g,H)=(g,f,H^{-1})$ where $H^{-1}$ is the inverse of the $2$-cell $H$ in the $\mathrm{Hom}$-groupoid $\mathrm{Hom}(X,Y)$. 
\end{proof}

\begin{lem}\label{DefoTraCatLem}
	A deformation category is a track category if and only if the following equations hold:
	\begin{align}
		H \sq\varepsilon(f)	&=H	&\text{where } f=\partial^-H,\notag\\
		\op(H)\sq H				&=\varepsilon(f)\notag\\
		(H\sq G)\sq F&= H\sq (G\sq F),\notag\\
		HH'&=H\varepsilon(g')\sq \varepsilon(f)H'	&\text{where }g'=\partial^+H'\tag{$*$}.
	\end{align}
\end{lem}

Here the last equation is required by the horizontal composition of tracks. The category of track categories is a full
subcategory of the category of deformation categories.

\begin{lem}
	A deformation category $(\cat{C}_0,\cat{C}_1)$ yields a natural equivalence relation $\simeq$ on the category
	$\cat{C}_0$ so that the homotopy category $\cat{C}_0\modu$ is defined.
\end{lem}

\begin{proof}
	For morphisms $f,g:X\to Y$ in $\cat{C}_0$ we write $f\simeq g$ if and only if there exists $H:X\to Y$ in $\cat{C}_1$
	with $\partial^-H=f$ and $\partial^+H=g$.
\end{proof}

\begin{ex}
	Let $B$ be a point and let $A$ be the empty set. Then $I_AB$ is the interval $I$. The category
	$\cat{C}_0=\cat{C}(\pt)$ coincides with $\cat{Top}^*$ and the morphisms in $\cat{C}_1=\cat{C}(I)$ are
	homotopies of pointed maps. The category $\cat{C}_0\modu$ coincided with the homotopy category $\cat{Top}^*\modu$.
	Homotopies between homotopies yield a natural equivalence relation $\simeq$ on $\cat{C}_1=\cat{C}(I)$ such that
	$(\cat{C}_0,\cat{C}_1\modu)$ is the track category associated to $\cat{Top}^*\modu$. This is a quotient of the
	deformation category $(\cat{C}_0,\cat{C}_1)$.
\end{ex}

A track category $\cat{C}=(\cat{C}_0,\cat{C}_1)$ is \emph{abelian} if all automorphism groups in hom-groupoids are
abelian groups. We denote such automorphism groups by $\mathrm{Aut}_{\square}(f)$ where $f$ is a morphism in $\cat{C}_0$ and
$H\in\mathrm{Aut}_{\square}(f)$ is of the form $H:f\simeq f$.

\begin{ex}\label{AbTraCatEx1}
	Let $\mathfrak{X}$ be a class of co-$H$-groups in $\cat{C}=\cat{Top}^*$ or let $\mathfrak{X}$ be a class of $H$-groups
	in $\cat{Top}^*$. Let $\cat{Top}^*[ [ \mathfrak{X} ] ]$ be the track subcategory
	of $(\cat{C}(\pt), \cat{C}(I)\modu)$ consisting of objects in $\mathfrak{X}$. Then 
	$\cat{Top}^*[ [ \mathfrak{X} ] ]$ is an abelian track category.
\end{ex}

The example has a generalization concerning ``under'' and ``over'' categories respectively which play the role of
``left'' resp. ``right'' modules.

\begin{defn}
	Let $\cat{C}$ be a category and let $\mathfrak{X}$ be a class of objects in $\cat{C}$ yielding the full subcategory
	$\cat{C}\{ \mathfrak{X} \}$ of $\cat{C}$. Let $W$ be an object in $\cat{C}$. Then the \emph{under
	category $\cat{C}\{W\to \mathfrak{X} \}$} consists of the objects in $\mathfrak{X}$ and the object $W$.
	Morphisms are all morphisms in $\cat{C}\{ \mathfrak{X} \}$ and all morphisms $W\to Y$ in $\cat{C}$ with
	$Y\in \mathfrak{X}$. Also the identity of $W$ is a morphism of the under category. The \emph{over category
	$\cat{C}\{ \mathfrak{X}\to W \}$} is defined in a dual way as a subcategory of $\cat{C}$.
\end{defn}

\begin{ex}\label{AbTraCatEx2}
	Let $\mathfrak{X}$ be a class of co-$H$-groups in $\cat{Top}^*$ and let $\cat{Top}^*[[ 
	\mathfrak{X}\to W
	]]$ be the track category given by the over category $\cat{Top}^*\{ \mathfrak{X}\to W \}$. Then
	 $\cat{Top}^*[ [ \mathfrak{X}\to W] ]$ is an abelian track category. Dually for a class
	 $\mathfrak{X}$ of $H$-groups the under track category $\cat{Top}^*[ [W\to \mathfrak{X}	] ]$ is
	 abelian. These are again full track subcategories of $(\cat{C}(\pt),\cat{C}(I)\modu)$.
\end{ex}

Given a category \cat{K} we define the \emph{category $F\cat{K}$ of factorizations} in $\cat{K}$. Objects in $F\cat{K}$
are morphisms $f$ in \cat{K} and a morphism $f\to g$ in $F\cat{K}$ is a pair $(a,b)$ of morphisms with $bfa=g$. We have
morphisms $(b,1):f\to bf$ and $(1,a):f\to fa$. A functor $D:F\cat{K}\to\cat{Ab}$ is termed a \emph{natural system} of
abelian groups. We write $D_f=D(f)$ and $(b,1)_*=b_*:D_f\to D_{bf}$ and $(1,a)_*=a^*:D_f\to D_{fa}$. In \cite{BW} the
\emph{cohomology} $H^n(\cat{K},D)$ is defined.

Let $(\cat{C}_0,\cat{C}_1)$ be an abelian track category. Then there is an \emph{associated} natural system $D$ on
$\cat{C}_0\modu$ together with a natural isomorphism of abelian groups
\begin{equation*}
	\sigma:D_{\{f  \}}\cong\mathrm{Aut}_{\square}(f)
\end{equation*}
for $f$ in $\cat{C}_0$, see \cite{BJCl}. Here $\{ f \}$ denotes the homotopy class of $f$ in
$\cat{C}_0\modu$. We call $(\cat{C}_0,\cat{C}_1)$ a \emph{linear track extension} of $\cat{C}_0\modu$ by $D$. It is
proved in \cite{BD}, \cite{P} that equivalence classes of such linear track extensions are in 1-1 correspondence with
elements in the cohomology $H^3(\cat{C}_0\modu,D)$, \cite{BW}.

We now describe the natural system associated to the abelian track category in the above examples of \ref{AbTraCatEx1} 
and \ref{AbTraCatEx2}. Here $\mathfrak{X}$
is a class of co-$H$-groups or of $H$-groups. For pointed spaces $X,Y$ let $X\vee Y$ be the coproduct in $\cat{Top}^*$
with inclusions $i_1,i_2$ and let $X\times Y$ be the product of spaces with projections $p_1,p_2$. For a map $f:X\to
Y$ in $\cat{Top}^*\modu$ with $X,Y\in \mathfrak{X}$ we define in the co-$H$-case
\begin{equation*}
	\begin{aligned}
		\nabla f&:X\to Y\vee Y\\ 
		\nabla f&= -i_2 f+f(i_2+i_1).
	\end{aligned}
	\begin{gathered}
		\text{in }\cat{Top}^*\modu,\\
		\
\end{gathered}
\end{equation*}
	Since $p_2(\nabla f)=0$ the \emph{partial suspension} ($n\geq 1$)
\begin{equation*}
	E^n\nabla f:\Sigma^nX\to (\Sigma^nY)\vee Y
\end{equation*}
is defined, see \cite{BOb}. In the $H$-case we get
\begin{equation*}
	\begin{aligned}
		\nabla f&:X\times X \to Y\\
		\nabla f&=-fp_2+(p_2+p_1)f.
	\end{aligned}
	\begin{gathered}
		\text{in }\cat{Top}^*\modu,\\
		\
	\end{gathered}
\end{equation*}
Since $(\nabla f)i_1 =0$ the \emph{partial loop operation} ($n\geq 1$)
\begin{equation*}
	L^n\nabla f:(\Omega^nX)\times X\to Y
\end{equation*}
is defined, see \cite{BOb}. In \cite{BOb} we describe rules to compute $E^n\nabla f$ and $L^n\nabla f$ explicitly, see
also \cite{BAl}.

\begin{defn}
	Let $\mathfrak{X}$ be a class of co-$H$-groups. Then we define a natural system $D_{\Sigma}^n$ on the over category
	$\cat{K}=\cat{Top}^*\{ \mathfrak{X}\to W\}\modu$. If $\mathfrak{X}$ is a class of $H$-groups we define a
	natural system $D_{\Omega}^n$ on the under category $\cat{K}=\cat{Top}^*\{
	W\to\mathfrak{X}\}\modu$. 
	Consider maps
	\begin{equation*}
		X'\xrightarrow{a}X\xrightarrow{f}Y\xrightarrow{b}Y'
\end{equation*}
in $\cat{K}$. In the co-$H$-case we define
\begin{equation*}
	D_{\Sigma}^n(f)=[ \Sigma^n X,Y ]
\end{equation*}
and for $\alpha\in D_{\Sigma}^n(f)$ let $b_*(\alpha)=b\alpha$ and $a^*(\alpha)=(\alpha,f)(E^n\nabla a)$. 
Moreover in the $H$-case we define
\begin{equation*}
	D_{\Omega}^n(f)=[ X,\Omega^nY ]
\end{equation*}
and for $\beta\in D_{\Omega}^n(f)$, let $b_*(\beta)=(L^n\nabla b)(\beta,f)$ and $\alpha^*(\beta)=\beta a$.
\end{defn}

We point out that in general $D_{\Sigma}^n$ and $D_{\Omega}^n$ are not given by the bifunctors 	
$[\Sigma^n X,Y ]$ and $[ X,\Omega^nY ]$ respectively. In \cite{BJ}
we prove for the examples above:

\begin{prop}
	The natural systems associated to the abelian track categories $\cat{Top}^*[ [ \mathfrak{X}\to W]
	]$ and $\cat{Top}^*[ [W\to \mathfrak{X}] ]$ are $D_{\Sigma}^1$ and $D_{\Omega}^1$
	respectively.
\end{prop}

Next we describe \emph{triple Toda brackets} in abelian track categories $(\cat{C}_0,\cat{C}_1)$ with natural system $D$
on $\cat{C}_0\modu$. Let
\begin{equation*}
	W\xleftarrow{\alpha}X\xleftarrow{\beta}Y\xleftarrow{\gamma}Z
\end{equation*}
be morphisms in $\cat{C}_0\modu$ with $\alpha\beta=0$ and $\beta\gamma=0$. Here $0=0(Y,W):Y\to*\to W$ is the
\emph{trivial map} given by the zero object $*$. Then we can choose representatives $f,g,h$ of $\alpha,\beta,\gamma$
respectively and $H:Y\to W, G:Z\to X$ in $\cat{C}_1$ with
\begin{equation*}
	\partial^+H=0,\ \partial^-H=fg,\ \partial^+G=0,\ \partial^-G=gh
\end{equation*}
Hence the element
\begin{equation*}
	c(H,G)=(H\varepsilon(g))\sq(\varepsilon(f)G^{\op})\in \mathrm{Aut}_{\square}(0(Z,W))=D_{0(Z,W)}
\end{equation*}
is defined. The collection of all such elements yields the Toda bracket in the quotient group
\begin{equation*}
	\left<\alpha,\beta,\gamma\right>\in D_{0(Z,W)}\left/\gamma^*D_{0(Y,W)}+\alpha_*D_{0(Z,X)}\right.
\end{equation*}

\section{The indexing set of balls}\label{IndSet}

We introduce the indexing set of balls which generalizes the set of cubes, but has more properties concerning relative
cylinders. The category $\mathbb{B}$ of such balls yields the $\mathbb{B}$-sets generalizing cubical sets. We consider
$\mathbb{B}$-categories.

A \emph{ball} $B$ of dimension $n$ is a finite regular CW-complex together with a subcomplex $\partial B$ such that
$(B,\partial B)$ is homeomorphic to the Euclidean ball $(E^n,S^{n-1})$ with $E^n=\{ x\in\mathbb{R}^n, \lVert x\rVert \leq 1 \}$
and $S^{n-1}=\{ x\in\mathbb{R}^n, \lVert x\rVert = 1\}$. We say that the ball $B$ is \emph{elementary} if the
CW-complex $B$ has exactly one open $n$-cell. If $B$ is a ball then the relative cylinder (as a quotient complex of
$I\times B$)
\begin{equation*}
	J(B)=I_{\partial B}B
\end{equation*}
is a ball of dimension $\dim(B)+1$. Moreover if $B$ and $B'$ are balls then the \emph{product} CW-complex
$B\times B'$ is a ball of dimension $\dim(B)+\dim(B')$. 
A \emph{ball pair} $(B,A)$ is a ball $B$ together with a subcomplex $A\subset \partial B$ which is a ball such that there
exists a homeomorphism of pairs $(J(A),A^-)\approx (B,A)$ extending the identity of $A$.
If $A_{\op}$ is the closure of $\partial B - A$ then also $(B,A_{\op})$ is a ball pair which we call the \emph{opposite} of
$(B,A)$. If $(B,A)$ and $(B',A)$ are ball pairs then the \emph{gluing}
\begin{equation*}
	B\cup B'=B \cup_{A} B'
\end{equation*}
is again a ball of dimension $\dim(B)=\dim(B')$.

\begin{defn}
	The \emph{indexing set $\mathbb{B}$ of balls} is the smallest set of balls with the following properties:
	\begin{enumerate}
		\item $\pt\in \mathbb{B}$,
		\item If $B\in \mathbb{B}$ then $J(B)\in \mathbb{B}$,
		\item If $B,B'\in \mathbb{B}$ then $B\times B'\in \mathbb{B}$,
		\item If $(B,A)$ is a ball pair with $B\in \mathbb{B}$ then $A\in \mathbb{B}$,
		\item If $(B,A)$ and $(B',A)$ are ball pairs with $B,B'\in\mathbb{B}$ then $B\cup_A B'\in \mathbb{B}$. Conversely if
			$B\cup_A B'\in\mathbb{B}$ then also $B,B'\in \mathbb{B}$.
	\end{enumerate}
\end{defn}

We now describe some examples of balls in $\mathbb{B}$. Since $\{ \pt \}\in\mathbb{B}$ also the
\emph{interval} from 0 to 1 is a ball
\begin{equation*}
	I=[ 0,1 ]=J(\pt)\in\mathbb{B}.
\end{equation*}
Moreover inductively $J^n,I^n\in\mathbb{B},\ n\geq 0$, where $J^0=I^0=\pt$ and $J^1=I^1=I$ and for $n\geq 1$
\begin{align*}
	J^{n+1}&=J(J^n)\\
	I^{n+1}&=I\times I^n
\end{align*}
Here $I^n$ is the \emph{$n$-cube} and $J^n$ is the \emph{$n$-globe} with hemispheres $i^-(J^{n-1}), i^+(J^{n-1})$.

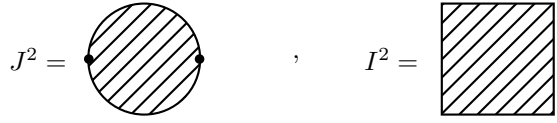
\begin{center}
\begin{pspicture}(-3,-1.3)(3,1.3)
	\rput(-3.4,0){$J^2=$}
	\pscircle[fillstyle=hlines](-2,0){0.75}
	\psdots(-1.27,0)
	\psdots(-2.73,0)
	\rput(0,0){,}
	\rput(1.25,0){$I^2=$}
	\psframe[fillstyle=hlines](1.9,-0.75)(3.4,0.75)
\end{pspicture}
\end{center}

Let $\mathbb{B}^n$ be the set of $n$-dimensional balls in $\mathbb{B}$. Then $\mathbb{B}^0$ contains only the point.
Moreover $\mathbb{B}^1$ is given by all \emph{long intervals} $[ 0,n ]\subset \mathbb{R}$ with 0-cells given
by $0,1,\ldots,n\in\mathbb{R}$. Next $\mathbb{B}^2$ consists of balls $J^2,I^2,\Delta^2$ (where $\Delta^2$ is the
2-simplex) and the balls obtained by a regular sequence of gluings of these balls, for example
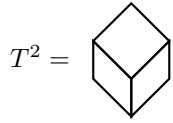
\begin{center}
	\begin{pspicture}(-1.5,-0.8)(1.5,1.3)
		\rput(-1.2,0.3){$T^2=$}
		\pspolygon(0,0)(-0.5,0.5)(0,1)(0.5,0.5)
		\pspolygon(0,-0.5)(0,0)(0.5,0.5)(0.5,0)
		\pspolygon(0,-0.5)(0,0)(-0.5,0.5)(-0.5,0)
	\end{pspicture}
\end{center}
is such a ball, see section \ref{BCatTen} below.

Let $\mathbb{B}$ be the full subcategory of $\cat{Top}$ with objects in the indexing set of balls. Let
$\mathbb{B}_{\partial}$ be the subcategory of pair maps $f:(B,\partial B)\to (A,\partial A)$ with $A,B\in \mathbb{B}$ and let
$B_{\partial}\modu$ be the homotopy category relative the boundary with $f\simeq g\Leftrightarrow f|\partial
B=g|\partial B$ and $f\simeq g\rel \partial B$.

\begin{remark}
	Given a small category \cat{D} a \emph{\cat{D}-object} $X$ in a category $\cat{U}$ is a contravariant functor
	$X:\cat{D}\to \cat{U}$. This is a $\cat{D}$-set if \cat{U} is the category of sets. Moreover let $\calO$ be a fixed
	class of objects with $*\in \calO$ and let $\cat{cat}(\calO)$ be the category of categories for which the class of
	objects is $\calO$ with zero object $\ast$. Morphisms are functors which are the identity on $\calO$. Then a
	\emph{$\cat{D}$-category} is a \cat{D}-object in $\cat{cat}(\calO)$. For the category $\mathbb{B}$ of balls we shall
	considers $\mathbb{B}$-sets and $\mathbb{B}$-categories. They generalize cubical sets and cubical categories
	respectively since the category of cubes is a subcategory of $\mathbb{B}$.
\end{remark}

For $B\in\mathbb{B}$ we have the \emph{category $\cat{C}(B)$ of $B$-maps} defined in \ref{CatDef}. Moreover
$f:B\to A$ in $\mathbb{B}$ induces a functor $f^*:\cat{C}(A)\to \cat{C}(B)$ in $\cat{cat}(\calO)$ where $\calO$ is the class
of pointed spaces. Hence
\begin{equation*}
	\cat{C}=\{ \cat{C}(B),B\in\mathbb{B} \}
\end{equation*}
is a $\mathbb{B}$-category which we call the \emph{topological $\mathbb{B}$-category}. The morphism sets $\cat{C}(X,Y)$
of $B$-maps $X\to Y$ for $B\in \mathbb{B}$ form a $\mathbb{B}$-set.

For any $\mathbb{B}$-category $\cat{C}$ and a ball pair $(B,A)$ let
\begin{equation*}
	\partial_A=(i_A)^*:\cat{C}(B)\to\cat{C}(A)
\end{equation*}
be induced by the inclusion $i_A:A\subset B$. We also write $\partial_Af=f|A$.
For example for $(JB,B^{\pm})$ we get
\begin{equation*}
	\partial^{\pm}=\partial_{B^{\pm}}:\cat{C}(JB)\to \cat{C}(B)
\end{equation*}
and for $H$ in $\cat{C}(JB)$ we write $H:f\simeq g$ if $\partial^-H=f$ and $\partial^+H=g$.

\section{$\mathbb{B}$-categories with unions}\label{BCatUni}

Since the union of balls is defined in $\mathbb{B}$ we can consider unions and gluings in a $\mathbb{B}$-category. 

Let \cat{C} be the $\mathbb{B}$-category of $B$-maps defined in section \ref{IndSet}. Given ball pairs $(B,A)$ and
$(B',A)$ with $B,B'\in \mathbb{B}$ one obtains the \emph{gluing functor}
\begin{equation*}
	\cup:\cat{C}(B)\times_{\cat{C}(A)}\cat{C}(B')\to \cat{C}(B\cup_A B').
\end{equation*}
Here the pullback category is given by pairs $(F:X\to Y)\in\cat{C}(B), (G:X\to Y)\in\cat{C}(B')$ with $\partial_A
F=\partial_A G$ and $F\cup G$ is defined by
\begin{equation*}
	(B\cup_A B')\times X=(B\times X)_{A\times X}(B'\times X)\xrightarrow{F\cup G}Y.
\end{equation*}
The gluing functor is \emph{natural} with respect to pair maps $f:(B,A)\to (B_1,A_1), g:(B',A)\to(B_1',A_1)$ in
$\mathbb{B}$ which coincide on $A,\ f|A=g|A$. Then
\begin{equation*}
	f^*F\cup g^*G=(f\cup g)^*(F\cup G)
\end{equation*}
We now describe the gluing rule in the $\mathbb{B}$-category \cat{C}.

Let $\overline{B}$ be a ball and let $M$ be a set of balls $B$ which are subcomplexes of $\overline{B}$ and
$\dim B=\dim\overline{B}$. Moreover assume that for $B,B'\in M$ with $B\neq B'$ we have $(B-\partial B)\cap
(B'-\partial B')=\emptyset$ and 
\begin{equation*}
	\bigcup_{B\in M}B=\overline{B}
\end{equation*}
Then we say that $\overline{B}$ is a \emph{union of balls} in $M$. A \emph{regular sequence} of balls in $\overline{B}$ is a
bijection $\{ B_1,\ldots,B_k \}\approx M$ (where $k$ is the number of elements in $M$) such that
\begin{equation*}
	B_{\leq i}=\bigcup_{j\leq i}B_j,\quad 1\leq i<k,
\end{equation*}
is a ball and $(B_{\leq i},A_i), (B_{i+1},A_i)$ are ball pairs with $A_i=B_{\leq i}\cap B_{i+1}$. If $B$-maps $f_B:X\to
Y$ in $\cat{C}(B)$ are given for $B\in M$ then each regular sequence in $\overline{B}$ yields the iterated gluing operation 
\begin{equation*}
	\left( \ldots \left( \left( f_{B_1}\cup f_{B_2} \right)\cup f_{B_3} \right)\ldots \cup f_{B_k} \right)
\end{equation*}
which is a $\overline{B}$-map $X\to Y$. The \emph{gluing rule} is the fact that this $\overline{B}$-map is independent of the
choice of the regular sequence in $\overline{B}$.

\begin{lem}
	Let \cat{C} be a $\mathbb{B}$-category with unions. Then for $B\in \mathbb{B}$ the pair
	$(\cat{C}(B),\cat{C}(JB))$ is
	a deformation category with structure $\varepsilon,\op,\square$. This yields the homotopy category
	$\cat{C}[B]=\cat{C}(B)\modu$ and $\cat{C}[_-]=\{ \cat{C}[B], B\in\mathbb{B} \}$ is a
	$( \mathbb{B}_{\partial}\modu )$-category.
\end{lem}

\begin{proof}
	The maps $\varepsilon':JB\to B, i^{\pm}:B\to JB, \op':JB\to JB, \square':JB\to JB\cup JB$ are maps in $\mathbb{B}$. They
	induce $\varepsilon=(\varepsilon')^*,\ \op=(\op')^*,\ \partial^{\pm}=(i^{\pm})^*$. Moreover $\square'$ yields $\square$ by the
	union property, that is 
	\begin{equation*}
	(\square')^*(F\cup G)=F\sq G.
	\end{equation*}
\end{proof}

Since $i^+i_A=i^-i_A$ we have $\partial_A\partial^+=\partial_A\partial^-$ and therefore there is a commutative diagram
\begin{equation*}
	\xymatrix{
	\cat{C}(B)\ar[rd]_{\partial_A}\ar[rr]^q	&	&\cat{C}[B]\ar[ld]^{\partial_A}\\
	&	\cat{C}(A)	&
	}
\end{equation*}
for each ball pair $(B,A)$. Here $q$ is the quotient map.

The topological $\mathbb{B}$-category \cat{C} has the property termed \emph{quotient property} that there is also a well 
defined gluing functor $\cup$ on homotopy categories $\cat{C}[B]=\cat{B}\modu$, that is, the following diagram commutes 
where $q$ is the quotient functor.
\begin{equation*}
	\xymatrix{
	\cat{C}(B)\times_{\cat{C}(A)}\cat{C}(B')\ar[rr]^{\cup}\ar[dd]^{q\times q}	&	&	\cat{C}(B\cup_A B')\ar[dd]^q \\
	\\
	\cat{C}[B]\times_{\cat{C}(A)}\cat{C}[B']\ar[rr]^{\cup}&	&\cat{C}[B\cup_AB']
	}
\end{equation*}

\begin{defn}
	A \emph{$\mathbb{B}$-category with unions} is a $\mathbb{B}$-category \cat{C} with a gluing functor, which
	is natural and satisfies the union rule and the quotient property.
\end{defn}

The topological $\mathbb{B}$-category is an example of a $\mathbb{B}$-category with unions.

\begin{remark}
	The quotient property can not be deduced from the properties of an abstract $\mathbb{B}$-category since
	$JJB\cup_B JJB$ is not a ball.
\end{remark}

\begin{lem}\label{BJBDefoCatLem}
	Let \cat{C} be a $\mathbb{B}$-category with unions. Then for $B\in\mathbb{B}$ the pair $(\cat{C}(B),\cat{C}[JB])$ is a
	deformation category which is a quotient of the deformation category $( \cat{C}(B),\cat{C}(JB) )$. Moreover for objects
	$X,Y$ the morphism sets
	\begin{equation*}
		(\cat{C}(B)(X,Y),\cat{C}[JB](X,Y))
	\end{equation*}
	form a groupoid.
\end{lem}

For the proof below we use thin fillers obtained as follows. If $s,t:B\to B'$ are maps in $\mathbb{B}$ with $s|\partial
B=t|\partial B$ then there exists $T:JB\to B'$ with $\partial^+T=t,\ \partial^-T=s$. We can find $T$ since $B'$ is
contractible. Hence $T^*:\cat{C}(B')\to \cat{C}(JB)$ satisfies $\partial^-T^*=s^*,\ \partial^+T^*=t^*$. We call
$T^*(x):s^*(x)\simeq t^*(x)$ a \emph{thin filler} for the pair $(s^*(x),t^*(x))$.

\begin{proof}
	We have to show that the equations in Lemma \ref{DefoTraCatLem} are satisfied, that is, there exist elements
	$\overline{\varepsilon},\overline{\op},\overline{\square}$ in $\cat{C}(JJ(B))$ which are homotopies
	\begin{gather*}
		\overline{\varepsilon}:H\sq \varepsilon(f)\simeq H \quad \text{for }\partial^-H=f \text{ in }\cat{C}(B)\\
		\overline{\op}:\op(H)\sq H\simeq \varepsilon(f)\\
		\overline{\square}:(H\sq G)\sq F\simeq H\sq (G\sq F)
	\end{gather*}
	We obtain these homotopies by use of thin fillers. We consider the diagram
	\begin{equation*}
	\xymatrix{
	JB\ar[d]^{\square'}\ar[drr]^1&&\\
	JB\cup JB\ar[rr]_{i^-\varepsilon'\cup 1}&&JB
	}
	\end{equation*}
	Here $\square'$ induces $\square$ and $\varepsilon':JB\to B$ induces $\varepsilon$. Using a thin filler for the induced
	operators one gets $\overline{\varepsilon}$. In a similar way one gets $\overline{\op},\overline{\square}$.
\end{proof}

\section{$\mathbb{B}$-categories with $\otimes$-products}\label{BCatTen}
Let \cat{C} be the topological $\mathbb{B}$-category of $B$-maps in section \ref{IndSet}. Since products of balls are
defined in $\mathbb{B}$ we can define $\otimes$-products in $\mathbb{B}$-categories as follows.

For $B,B'\in \mathbb{B}$ and for a $B'$-map $g:Y\to Z$ and a $B$-map $f:X\to Y$ we have the $B'\times B$-map $g\otimes
f:X\to Z$ induced by the composite
\begin{equation*}
	B'\times B\times X\xrightarrow{B'\times f}B'\times Y\xrightarrow{g}Z.
\end{equation*}
This yields the $\otimes$-product in $\cat{C}$ which has properties as described in the next Lemmas.

\begin{lem}
	The operation $\otimes$ has the following properties:
	\begin{enumerate}
		\item $\otimes$ is associative.
		\item For $B$-maps $f:X\to Y,g:Y\to Z$ the composite $gf$ in $\cat{C}(B)$ satisfies
			\begin{equation*}
				\Delta^*(g\otimes f)=gf
			\end{equation*}
			where $\Delta^*:\cat{C}(B\times B)\to \cat{C}(B)$ is induced by the diagonal $B\to B\times B$.
		\item Let $0_{X,Y}^B$ be the trivial map $X\to *\to Y$ in $\cat{C}(B)$. Then we have for $(g:Y\to Z)\in \cat{C}(B'),
			(f:X\to Y)\in \cat{C}(B)$,
			\begin{align*}
				g\otimes 0_{X,Y}^B &= 0_{X,Z}^{B\times B'},\\
				0_{Y,Z}^{B'}\otimes f&=0_{X,Z}^{B\times B}.
			\end{align*}
	\end{enumerate}
\end{lem}

The point is the final object in $\mathbb{B}$ and the unique map $\varepsilon_0':B\to \pt$ induces the functor
\begin{equation*}
	\varepsilon_0:\cat{C}(\pt)\to \cat{C}(B)
\end{equation*}
with $f^*\varepsilon_0=\varepsilon_0$ for all $f:B\to B'$ in $\mathbb{B}$.

\begin{lem}
	For $g:Y\to Z$ in $\cat{C}(B')$ and $f:X\to Y$ in $\cat{C}(B)$ we have $g\otimes f:X\to Z$ in $\cat{C}(B'\times B)$. If
	$B'$ is the point then $f\otimes g=f\varepsilon_0(g)$ in $\cat{C}(B)$ and if $B$ is the point then $f\otimes
	g=\varepsilon_0(f)g$ in $\cat{C}(B')$.
\end{lem}

Moreover we have the following \emph{naturality of $\otimes$-products}.

\begin{lem}
	For maps $g_1:B_1'\to B', f_1:B_1\to B$ in $\mathbb{B}$ one has
	\begin{equation*}
		(g_1^* g)\otimes (f_1^* f)=(f_1\times g_1)^*(g\otimes f)
	\end{equation*}
\end{lem}

In particular, for all ball pairs $(B,A)$ the product $(B'\times B, B'\times A)$ is again a ball pair and we have
\begin{equation*}
	g\otimes \partial_A f=\partial_{B'\times A}(g\otimes f)
\end{equation*}

\begin{defn}
	An abstract \emph{$\mathbb{B}$-category \cat{C} with $\otimes$-products} is defined by $\otimes$-pro\-ducts which
	satisfy the properties in the Lemmata above.	
\end{defn}

\section{$\mathbb{B}$-deformation categories}\label{BDefCat}
Let \cat{C} be the topological $\mathbb{B}$-category. Then we have seen that \cat{C} has unions and $\otimes$-products.
Moreover the following formulas are satisfied.

Let $g:Y\to Z$ in $\cat{C}(B')$ and $f:X\to Y$ in $\cat{C}(B)$. Then $g\otimes f:X\to  Z$ in $\cat{C}(B\times B')$ is
defined. Now let $B=B_1\cup B_2$ or $B'=B_1'\cup B_2'$ be unions of ball pairs. Then for $f=f_1\cup f_2$ one has the
\emph{compatibility of $\otimes$ and $\cup$}
\begin{equation*}
	g\otimes (f_1\cup f_2)=(g\otimes f_1)\cup (g\otimes f_2)
\end{equation*}
and for $g=g_1\cup g_2$ one has
\begin{equation*}
	(g_1\cup g_2)\otimes f=(g_1\otimes f)\cup (g_2\otimes f)
\end{equation*}

\begin{defn}
	An abstract \emph{$\mathbb{B}$-deformation category} \cat{C} is a $\mathbb{B}$-category \cat{C} with unions $\cup$
	and $\otimes$-products such that these formulas on the compatibility of $\otimes$ and $\cup$ hold.
\end{defn}

The next Lemma extends the corresponding Lemma \ref{BJBDefoCatLem}.

\begin{lem}
	Let \cat{C} be an abstract $\mathbb{B}$-deformation category. Then for $B\in \mathbb{B}$ the deformation category
	$(\cat{C}(B),\cat{C}[JB])$ is a quotient of $(\cat{C}(B),\cat{C}(JB))$ and this quotient is a track category.
\end{lem}

\begin{proof}
	We have to find $R$ in $\cat{C}(JJB)$ such that 
	\begin{equation*}
		R:HE\simeq (H\varepsilon(e'))\sq (\varepsilon(f)E)
	\end{equation*}
	where $H:f\simeq f',\ E:e\simeq e'$. Here $HE$ is the composite in the category $\cat{C}(JB)$. Now $R$ is
	obtained by the diagram
	\begin{equation*}
		\xymatrix{
		JB\ar[dr]^{\Delta}\ar[d]_{\square'}	&	\\
		JB\cup JB\ar[r]_{H'\cup E'}	& JB\times JB
		}
	\end{equation*}
	Here $\Delta$ is the diagonal and we define
	\begin{align*}
		E'=(i^-\varepsilon',1)&:JB\to JB\times JB,\\
		H'=(1,i^+\varepsilon')&:JB\to JB\times JB.
	\end{align*}
	Then $\Delta^*(H\otimes E)=HE$ and $( (H'\cup E')\square')^*(H\otimes E)=( H\varepsilon(e')
	)\sq (\varepsilon(f)E)$.
	Hence a thin filler yields $R$.
\end{proof}

Let \cat{C} be an abstract $\mathbb{B}$-deformation category and let $(B,A)$ be a ball pair in $\mathbb{B}$. Then we can
choose a map
\begin{equation*}
	\square'_A:B\to B\cup_A JA
\end{equation*}
which is the identity on the boundary. The union is the pushout of $i_A$ and $i^+$. We call $\square'_A$ the
\emph{action map} which induces the \emph{action functor}
\begin{align*}
	&\square_A:\cat{C}[B]\times_{\cat{C}(A)} \cat{C}[JA]\to \cat{C}[B]\\
	&\square_A(F,G)= F\sq_A G=(\square'_A)^*(F\cup G)
\end{align*}
Here $\square_A$ does not depend on the choice of $\square'_A$. For $B=JA$ the functor $\square_A=\square$ is given by the
track category $(\cat{C}(A),\cat{C}[JA])$ above.

\begin{lem}
	Let \cat{C} be an abstract $\mathbb{B}$-deformation category. Then $\square_A$ satisfies the equations
	\begin{align*}
		&F\sq_A \varepsilon(f)=F \quad\text{for }\partial_A F=f,\\
		&F\sq_A (G\sq H)=(F\sq_A G)\sq_A H,\\
		&\partial_{A_{\op}}(F\sq_A G)=\partial_{A_{\op}}F,\\
		&\partial_A(F\sq_A G)=\partial^-G.
	\end{align*}
	Moreover given $F\in\cat{C}[B]$ with $\partial_A F=f,\ \partial_{A_{\op}}F=g$, the group $\mathrm{Aut}_{\square}(f)$ in
	the track category $(\cat{C}(A),\cat{C}[JA])$ acts transitively and effectively on the set
	\begin{equation*}
		\{ F'\in \cat{C}[B], \partial_A F'=f, \partial_{A_{\op}}F'=g \}	
	\end{equation*}
	by $F'\sq_A \alpha$ for $\alpha\in\mathrm{Aut}_{\square}(f)$
\end{lem}

The Lemma describes the \emph{action property} of $\square_A$.

\section{Abelian $\mathbb{B}$-deformation categories}\label{AbeDef}

For all $B\in \mathbb{B}$ one has the \emph{projection functor}
\begin{equation*}
	\partial_0:\cat{C}(B)\to \cat{C}(\pt)\modu
\end{equation*}
which is induced by any map $f:\pt\to B$ and which by thin fillers is independent of this choice.

\begin{defn}
	An abstract $\mathbb{B}$-deformation category \cat{C} is \emph{pre-abelian} if the track categories
	\begin{equation*}
		(\cat{C}(B),\cat{C}[JB])
	\end{equation*}
	are abelian for all $B\in \mathbb{B}$ and if the associated natural systems $D_B$ are composites of the form
	\begin{equation*}
		D_B:F(\cat{C}(B)\modu)\xrightarrow{\partial_0}F(\cat{C}(\pt)\modu\xrightarrow{D^n}\cat{Ab}
	\end{equation*}
	Here $D^n$ with $n=\dim(B)$ depends only on $n$ and not on $B$.
\end{defn}

Now let \cat{C} be pre-abelian with natural systems $D^n, n\geq 1$. Then we have for a ball pair $(B,A)$ and $f:X\to Y$
in $\cat{C}(A)$ the group $\mathrm{Aut}_{\square}(f)$ in the track category $(\cat{C}(A),\cat{C}[JA])$ and we have the
isomorphism, $n=\dim(A)+1$, (see section \ref{TraCat})
\begin{equation*}
	\sigma_A=\sigma:D^n_{\partial_0 f}\cong \mathrm{Aut}_{\square}(f)
\end{equation*}
which is natural in $f$. Hence the $\square_A$-action of $\mathrm{Aut}_{\square}(f)$ on the set $\{ F\in
\cat{C}[B], \partial_A F=f\}$ yields an action of $\alpha\in D^n_{\partial_0 f}$ on this set by $F\sq_A \alpha =
F\sq_A \sigma(\alpha)$. Here we have $\partial_0f=\partial_0\partial_A F=\partial_0 F$.

\begin{defn}
	Let $\ori(B)$ be a manifold orientation of the topological manifold $B\in \mathbb{B}$. there are two such orientations,
	$\ori(B)$ and $\overline{\ori}(B)$. Now $\ori(B)$ yields a manifold orientation $\ori(\partial B)$ on the boundary $\partial
	B$. If $(B,A)$ is a ball pair $\ori(\partial B)$ yields by restriction a manifold orientation $\ori(A)$. Now the product
	$I\times A$, where $I$ is the oriented interval from 0 to 1, yields a manifold orientation $\ori(JA)$. If $\ori(B)\cup
	\ori(JA)$ define a manifold orientation of $B\cup_A JA$ we write $\varepsilon(B,A)=+1$ and $\varepsilon(B,A)=-1$
	otherwise.
\end{defn}

\begin{defn}
	Let \cat{C} be pre-abelian and consider for a ball pair $(B,A)$ and $F\in\cat{C}[B],\ \partial_A F=f$ the element
	\begin{equation*}
		F\sq_{\ori(A)}\alpha= F\sq_A(\varepsilon(B,A)\alpha)\in \cat{C}[B]
	\end{equation*}
	where $\alpha\in D^n_{\partial_0 F}$. We call \cat{C} an \emph{abelian} $\mathbb{B}$-deformation category if this
	element does not depend on the ball pair $(B,A)$ but only on $F, \ori(B)$ and $\alpha$. For the opposite orientation
	$\overline{\ori}(B)$ we have the formula
	\begin{equation*}
		F\sq_{\overline{\ori}(B)}\alpha=F\sq_{\ori(B)}(-\alpha)
	\end{equation*}
	Moreover the \emph{abelian union property} is satisfied, that is, for a union $\overline{B}=B_1\cup\ldots\cup B_k$ of
	balls and $\ori(B_i)$ induced by $\ori(\overline{B}),\ i=1,\ldots,k$, we have for $F_i\in\cat{C}[B_i]$ the formula
	\begin{equation*}
		F_1\cup\ldots\cup (F_i\sq_{\ori(B_i)}\alpha)\cup\ldots\cup F_k=(F_1\cup\ldots\cup F_k)\sq_{\ori(\overline{B})}\alpha
	\end{equation*}
	where $\partial_0(F_i)=\partial_0(F_1\cup\ldots\cup F_k)$.
\end{defn}

Let \cat{C} be the topological $\mathbb{B}$-category and let $\mathfrak{X}$ be a class of objects in $\cat{Top}^*$ and
let $W$ be a space in $\cat{Top}^*$. For each $B\in\mathbb{B}$ we obtain the over category
\begin{equation*}
	\cat{C}(B)\{ \mathfrak{X}\to W \}\subset \cat{C}(B)
\end{equation*}
and dually the under category $\cat{C}(B)\{ W\to \mathfrak{X} \}$. Then the collection of over categories
\begin{equation*}
	\cat{C}\{ \mathfrak{X}\to W \}=\{ \cat{C}(B)\{ \mathfrak{X}\to W \}, B\in\mathbb{B} \}
\end{equation*}
is again a $\mathbb{B}$-deformation category. The same holds for the under category $\cat{C}\{ W\to\mathfrak{X}
\}$.

\begin{prop}
	Let $\mathfrak{X}$ be a class of co-$H$-groups. Then $\cat{C}\{ \mathfrak{X}\to W \}$ is an abelian
	$\mathbb{B}$-deformation category with natural systems $D^n=D^n_{\Sigma}, n\geq 1$. Dually, if $\mathfrak{X}$ is a
	class of $H$-groups then $\cat{C}\{ W\to\mathfrak{X} \}$ is an abelian $\mathbb{B}$-deformation category
	with natural system $D^n=D^n_{\Omega}, n\geq 1$.
\end{prop}

Compare section \ref{TraCat} where $D^n_{\Sigma}$ and $D^n_{\Omega}$ are defined.

\section{Higher order track categories}\label{HigOrdTra}

Let $\cat{C}$ be a $\mathbb{B}$-deformation category. Then we define for $n\geq 1$ the collection of categories
$\cat{C}^n$, termed the $n$-\emph{truncation} of \cat{C}, by
\begin{equation*}
	\cat{C}^n(B)=
	\begin{cases} 
		\cat{C}(B)\modu=\cat{C}[B] & \text{if }\dim(B)=n,\\
		\cat{C}(B)& \text{if }\dim(B)<n
	\end{cases}
\end{equation*}
with $B\in\mathbb{B}$. Hence only in dimension $n$ we take the homotopy category.

The following definition of an ``$n$-th order track category'' is motivated by those properties of a
$\mathbb{B}$-deformation category which remain visible in the truncation $\cat{C}^n$. Moreover these properties should be
minimal to allow the definition of higher order Toda brackets and certain spectral sequences below (see section 
\ref{HigOrdCoh}), and they should be
fulfilled by algebraic examples given by chain complexes (see section \ref{TruChaAlg}).

Let $\mathbb{B}(n), n\geq 0$ be the subcategory of $\mathbb{B}$ consisting of balls $B$ with $\dim(B)\leq n$ and
which is generated by the maps:
\begin{itemize}
	\item ball pair inclusions $i_A:A\subset B$,
	\item projections $\epsilon':JA\to A$,
	\item trivial maps $\varepsilon'_0:B\to \pt$.
\end{itemize}
Let $\mathbb{B}_{\partial}(n), n \geq 0$, be the subcategory of $\mathbb{B}$ consisting of balls $B$ with
$\dim(B)\leq n$ and which is generated by the following maps:
\begin{itemize}
	\item The \emph{opposite map} $\op':JA\to JA$,
	\item For a ball pair $(B,A)$ we have the ball $B\cup_A JA$ by pushout of $i_A, i^+$. The boundary satisfies
		$\partial(B\cup_A JA)=A_{\op}\cup A^-=A_{\op}\cup A=\partial B$. Then there is a map
		\begin{equation*}
			\sq'_A:B\to B\cup_A JA\quad\text{in }\mathbb{B}_{\partial}(n)
		\end{equation*}
		which is the identity on the boundary. We call $\sq_A'$ an \emph{action map}. Also $1\cup\varepsilon:B\cup_A JA\to
		B$ is a map in $\mathbb{B}_{\partial}(n)$.
	\item There is a map $h_A:A\to A_{\op}$ which extends the identity on $\partial A_{\op}=\partial A$. This map is called a
		\emph{comparison map}.
\end{itemize}
All the maps of $\mathbb{B}_{\partial}(n)$ are isomorphisms in the homotopy category
$\mathbb{B}_{\partial}(n)\modu\subset \mathbb{B}_{\partial}\modu$.

As above we use only functors in $\cat{cat}^*(\calO)$ which are the identity on objects.

\begin{defn}
	An \emph{$n$-th order track category} \cat{K} is a $\mathbb{B}(n)$-category given by categories $\cat{K}(B)$ with zero
	object and functors $f^*:\cat{K}(B)\to \cat{K}(A)$ for $f:A\to B$ in $\mathbb{B}(n)$ such that the following
	properties hold:
	\begin{enumerate}
		\item The relation $\simeq$ on $\cat{K}(B)$, defined by
			\begin{equation*}
				f\simeq g\Leftrightarrow 
				\begin{cases}
					f=g & \text{if }\dim(B)=n,\\
					\exists F\in \cat{K}(JB) \text{ with }\partial^-F=f,\ \partial^+F=g&\text{if }\text{dim}(B)<n,
				\end{cases}
			\end{equation*}
			is a \emph{natural equivalence relation}. Here the functor $\partial^{\pm}=(i^{\pm})^*$ is induced by the ball pair
			inclusion $i^{\pm}:B\subset JB$. 

			Let $\cat{K}[B]=\cat{K}(B)\modu$ be the homotopy category with
			quotient functor $q:\cat{K}(B)\to \cat{K}[B]$. Then $\partial_A=i_A^*$ induces $\partial_A:\cat{K}[B]\to \cat{K}(A)$
			with $\partial_A q=\partial_A$ since $\partial_A \partial^-=\partial_A\partial^+$.

		\item The categories $\cat{K}[B]$ have the structure of a $\mathbb{B}_{\partial}(n)$-\emph{category} given by functors
			$f^*:\cat{K}[B]\to \cat{K}[A]$ for $f:A\to B$ in $\mathbb{B}_{\partial}(n)$. Moreover $f,g:A\to B$ in
			$\mathbb{B}_{\partial}(n)$ with $f|\partial A=g|\partial A$ and
			$f\simeq g\rel \partial A$ satisfy $f^*=g^*$ so that $\cat{K}[_-]$ is actually a $\left<
			\mathbb{B}_{\partial}(n)\modu\right>$-category. The induced functors $f^*$ are compatible with the restriction
			functors $\partial_A$ (as in a $\mathbb{B}$-deformation category).
		\item For ball pairs $(B,A), (B',A)$ in $\mathbb{B}(n)$ a \emph{gluing functor} $\cup$ is given for
			which the following diagram commutes
			\begin{equation*}
				\xymatrix{
				\cat{K}(B)\times_{\cat{K}(A)}\cat{K}(B')\ar[r]^-{\cup}\ar[d]^{q\times q}		&\cat{K}(B\cup_A B')\ar[d]^q\\
				\cat{K}[B]\times_{\cat{K}(A)}\cat{K}[B']\ar[r]^-{\cup}	&\cat{K}[B\cup_A B']
				}
			\end{equation*}
			The gluing functor is natural with respect to maps in $\mathbb{B}(n)$, resp. $\mathbb{B}_{\partial}(n)$, and the
			gluing rule is satisfied, see section \ref{BCatUni}
		\item For a ball pair $(B,A)$ we have the \emph{action functor}
			\begin{equation*}
				\square_A:\cat{K}[B]\times_{\cat{K}(A)}\cat{K}[JA]\to\cat{K}[B]
			\end{equation*}
			defined by $\square_A(G,F)=(\square_A')^*(G\cup F)$. For $B=JA$ the tuple
			\begin{equation*}
				\left( \cat{K}(A),\cat{K}[JA],\varepsilon=q^*(\varepsilon')^*,\op=(\op')^*,\square=\square_A,\partial^{\pm} \right)
			\end{equation*}
			is a \emph{track category} with associated homotopy category $\cat{K}[A]$. The action functor has the action
			property in section \ref{BDefCat}.
		\item Given a ball pair $(B,A)$ and $f$ in $\cat{K}(A)$. Then $f\simeq 0$ if and only if there exists $F$ in
			$\cat{K}(B)$ with $\partial_A F=f$ and $\partial_{A_{\op}} F=0$. This is the \emph{extension property} of \cat{K}.
		\item For $B'\times B$ in $\mathbb{B}(n)$ and $g:Y\to Z$ in $\cat{K}(B')$ and $f:X\to Y$ in $\cat{K}(B)$
			the $\otimes$-\emph{product}
			\begin{equation*}
				g\otimes f:X\to Y\text{ in }\cat{K}(B'\times B)
			\end{equation*}
			is given. We also write $q(g\otimes f)=g\otimes f$ in $\cat{K}[B\times B']$. The tensor product is associative and
			natural and $g\otimes 0=0,\ 0\otimes f=0$. Moreover for $\varepsilon_0=(\varepsilon_0')^*$ we have $f\otimes g =
			f\varepsilon_0(g)$ if $B'=\pt$ and $f\otimes g=\varepsilon_0(f)g$ if $B=\pt$. Also $g\otimes\partial_A
			f=\partial_{B'\times A}(g\otimes f),\ (\partial_{A'} g)\otimes f=\partial_{A'\times B}(g\otimes f)$ and $\otimes$ and
			$\cup$ are compatible.
		\item Let $\dim(B\times B_1)\leq n-1$ for balls $B,B_1$ in $\mathbb{B}$. We consider the boundary of the
			product $JB\times JB_1$ which is the following union of balls
			\begin{align*}
				\partial(JB\times JB_1)&=U_{\op}\cup U\\
				U&=B^+\times JB_1\cup JB\times B_1^+\\
				U_{\op}&=B^-\times JB_1\cup JB\times B_1^-
			\end{align*}
			and we use a comparison map $h_U:U\to U_{\op}$ in $\mathbb{B}_{\partial}(n)$. Let $G\in\cat{K}(JB),
			F\in\cat{K}(JB_1)$, then we have in $\cat{K}[U]$ the following \emph{boundary formula}, 
			\begin{equation}
				h_U^*\left( (\partial^-G)\otimes F\cup G\otimes (\partial^-F) \right)=(\partial^+G)\otimes F^{\op}\cup G^{\op}\otimes
				(\partial^+ F)\tag{$*$}
			\end{equation}
			For $\dim(B\times B_1)=0$ this corresponds to formula $(*)$ in a track category, compare the horizontal
			composition in section \ref{TraCat}.
	\end{enumerate}
\end{defn}

\begin{lem}
	Given a $\mathbb{B}$-deformation category \cat{C} one has for $n\geq 1$ an $n$-th order track category $\cat{C}^n$ by
	truncation of \cat{C}, see the definition of $\cat{C}^n(B)$ above. Similarly the \emph{truncation} $\cat{K}^{n-1}$ of
	an $n$-th order track category $\cat{K}$ is an $(n-1)$-st order track category.
\end{lem}

\begin{defn}
	A \emph{map $\alpha:\cat{K}\to \cat{K}'$ between $n$-th order track categories} is given by functors
	$\alpha:\cat{K}(B)\to \cat{K}'(B),\ \alpha:\cat{K}[B]\to\cat{K}'[B]$ with $q\alpha=\alpha q$. They are natural with
	respect to $\mathbb{B}(n)$ and $\mathbb{B}_{\partial}(n)$ respectively. Moreover $\alpha$ is compatible with gluing and
	$\otimes$-products. A map $\alpha$ is a \emph{weak equivalence} if $\alpha:\cat{K}[\pt]\to \cat{K}'[\pt]$ is an
	isomorphism and $\alpha$ induces isomorphisms of groups 
	\begin{equation*}
		\mathrm{Aut}_{\square}(f)\cong \mathrm{Aut}_{\square}(\alpha f)
	\end{equation*}
	for all $f$ in $\cat{K}(B),\ \dim(B)\leq n$.
\end{defn}

\begin{remark}
	There is also a notion of \emph{pseudo functor} and \emph{pseudo equivalence} between higher track categories, see
	\cite{BM} for $n=1$.
\end{remark}

Composites of inclusions of ball pairs yield for $B$ in $\mathbb{B}(n)$ maps $\{ \pt \}\to B$ in
$\mathbb{B}(n)$ which induce a well defined functor 
\begin{equation*}
	\partial_0:\cat{K}(B)\to\cat{K}(\pt)\modu=\cat{K}[\pt].
\end{equation*}

\begin{defn}
	An $n$-th order track category \cat{K} is \emph{abelian} if all track categories $(\cat{K}(B),\cat{K}[JB])$,
	$\dim(B)\leq n-1$, are abelian and the associated natural system $D^B$ is a composite of the form
	\begin{equation*}
		F(\cat{K}[B])\xrightarrow{\partial_0} F(\cat{K}[\pt])\xrightarrow{D^k}\cat{Ab}
	\end{equation*}
	where $D^k$ depends only on $k=\dim(B)$. Moreover the $\square_A$-action yields a well defined action
	$\square_{\ori(B)}$ which satisfies the abelian union property, see section \ref{AbeDef}.
\end{defn}

The truncation of the abelian $\mathbb{B}$-deformation categories in section \ref{AbeDef} yield examples of abelian
$n$-th order track categories.

\begin{ex}
	A first order track category \cat{K} is the same as a track category in section \ref{TraCat}. In fact, iterate
	composites of action maps yield a map $h:I\to [0,n]=I(n)$ inducing the isomorphism
	\begin{equation*}
		h^*:\cat{K}[I(n)]\approx\cat{K}(I(n))\xrightarrow{\approx}\cat{K}(I)\approx\cat{K}[I]
	\end{equation*}
	which does not depend on $h$. Now unions are defined by
	\begin{equation*}
		\xymatrix{
		\cat{K}(I(n))\times_{\partial}\cat{K}(I(m))\ar@2{-}[r]\ar[d]_{\cup}&\cat{K}(I)\times_{\partial}\cat{K}(I)\ar[d]^{\square}\\
		\cat{K}(I(n+m))\ar@2{-}[r] &\cat{K}(I)
		}
	\end{equation*}
	and $\otimes$-products are composites $\varepsilon(f)H, H\varepsilon(g)$ with $H$ in $\cat{K}(I)$ and $f,g$ in
	$\cat{K}(\pt)$. Of course \cat{K} is abelian iff the associated track category $(\cat{K}(\pt),\cat{K}[I])$ is abelian.
\end{ex}

\begin{remark}
	Let \cat{K} be an abelian $n$-th order track category. Then the truncation $\cat{K}^{n-1}$ is an abelian $(n-1)$-st
	order track category. We consider $\cat{K}$ as a \emph{linear track extension} of $\cat{K}^{n-1}$. The set of
	equivalence classes of such extensions is denoted by $H^{n+2}(\cat{K}^{n-1},D^n)$. This leads to a cohomology which
	for $n=1$ coincides with the cohomology in section \ref{TraCat}, \cite{BW}.
\end{remark}

\section{Higher order Toda brackets and higher order chain complexes}\label{HigOrdTod}

The properties of an abelian $n$-th order track category are chosen in such a way that it is possible to define higher
order Toda brackets generalizing the triple Toda brackets in a track category of section $\ref{TraCat}$.

\begin{prop}
	Let \cat{C} be an abelian $n$-th order track category with natural system $D^i, i=1,\ldots,n$ and let 
	\begin{equation*}
		Y=X_0\xleftarrow{\alpha_1}X_1\xleftarrow{\alpha_2}X_2\xleftarrow{}\cdots\xleftarrow{\alpha_{n+2}}X_{n+2}=X
	\end{equation*}
	be a sequence of morphisms $\alpha_i$ in $\cat{C}(\pt)\modu$. Then the higher Toda bracket is defined as a subset
	\begin{equation*}
		\left<\alpha_1,\ldots,\alpha_{n+2}\right>\subset D^n_{0(X,Y)}
	\end{equation*}
	For $n=1$ this is the triple Toda bracket in the abelian track category $(\cat{C}(\pt),\allowbreak\cat{C}(I)\modu)$. The set
	$\left<\alpha_1,\ldots,\alpha_{n+2}\right>$ is possibly empty. If $\cat{C}\sim\cat{K}$ is a weak equivalence then the
	brackets coincide, that is
	$\left<\alpha_1,\ldots,\alpha_{n+2}\right>_{\cat{C}}=\left<\alpha_1,\ldots,\alpha_{n+2}\right>_{\cat{K}}$.
\end{prop}

\begin{defn}\label{TrivBdDef}
	Let \cat{C} be an abelian $n$-th order track category. Let $B$ be a ball in $\mathbb{B}$ and let $F\in\cat{C}(B),\ 
	F:X\to Y$, be \emph{trivial on the boundary} (that is, for any ball pair $(B,A)$ we have $\partial_AF=0$ and
	$\partial_{A_{\op}}F=0$) then there is a unique element $\alpha\in D^i_{\partial_0 F},\ i=\dim(B)$, with 
	\begin{equation*}
		0\sq_{\ori(B)}\alpha=\{ F \} \text{ in }\cat{C}[B]
	\end{equation*}
	Here $\partial_{0}F=0(X,Y)$ in $\cat{C}[\pt]$. We call
	\begin{equation*}
		\ob(F)=\alpha\in D^i_{0(X,Y)}
	\end{equation*}
	the \emph{obstruction} associated to $F$. If $(E,B)$ is a ball pair the extension property shows that $\ob(F)=0$ if
	and only if there is $\overline{F}\in\cat{C}(E)$ with $\partial_B(\overline{F})=F$ and
	$\partial_{B_{\op}}(\overline{F})=0$.
\end{defn}

We now consider the ball pair $(I^{n+1},T^n)$ where $T^n$ is the union of all $n$-dimen\-sional faces of $I^{n+1}$ which
contain the origin $0=(0,\ldots,0)\in I^{n+1}$. Let $(I^{n+1},T^n_\op)$ be the opposite ball pair. Hence $T^n_\op$ is
the union of all faces of $I^{n+1}$ which contain the point $(1,\ldots,1)\in I^{n+1}$. We shall construct $F:X\to Y$ in
$\cat{C}(T^n)$ such that $F$ is trivial on the boundary and such that
\begin{equation*}
	\ob(F)\in\left<\alpha_1,\ldots,\alpha_{n+2}\right>
\end{equation*}
In fact, all possible choices of $F$ yield this way the set $\left<\alpha_1,\ldots,\alpha_{n+2}\right>$. If $F$ is not
constructable then this set is empty.

The element $\ob(F)$ is the zero element in the abelian group $D^n_{0(X,Y)}$ if and only if there exists $\overline{F}$
in $\cat{C}(I^{n+1})$ with $\overline{F}|T^n=F$ and $\overline{F}|T^n_\op=0(X,Y)\in\cat{C}(T^n_\op)$.

Here we call $\overline{F}$ a \emph{cubical extension} of $F$. We shall see that $F$ is constructible if and only if
inductively certain cubical extensions exist.

We start the induction by choosing representatives $f_i$ of the homotopy class $\alpha_i,\ i=1,\ldots,n+2$. Then we
choose
\begin{equation*}
	f_i^1\in\cat{C}(I),\quad f_i^1:f_if_{i+1}\simeq 0.
\end{equation*}
Assume now $f_i^k\in\cat{C}(I^k),\ 1\leq k<n,\ i\leq n-k+2$, are chosen. Then we define $f_i^{k+1}\in\cat{C}(I^{k+1})$
as follows, in fact, $f_i^{k+1}$ is a cubical extension of 
\begin{equation*}
	F_i^k\in\cat{C}(T^k)
\end{equation*}
where $F_i^k$ is obtained by the following gluing. The faces of $I^{k+1}$ in $T^k$ are of the form $I^k\times 0,\ 
I^{k-1}\times 0\times I,\ I^{k-2}\times 0\times I^2,\ldots,\ 0\times I^k$ and this is a regular sequence of balls in $T^k$.
Then $F_i^k$ is given by the restrictions
\begin{align*}
	&F_i^k|I^k\times 0= f_i^k \varepsilon_0(f_{i+k+1})\\
	&F_i^k|0\times I^k= \varepsilon_0(f_i)f_{i+1}^k\\
	&F_i^k|I^r\times 0\times I^k=f_i^r\otimes f_{i+r+1}^k,\quad r+s=k,\ 1\leq r<k.
\end{align*}

One can check that by the inductive construction this gluing, defining $F^k_i$, is well defined. Now let $F=F^n_0\in
\cat{C}(T^n)$. All possible choices of $F$ yield the set of elements $\ob(F)$ defining
$\left<\alpha_1,\ldots,\alpha_{n+2}\right>$.

We now consider \emph{higher order chain complexes}. Let \cat{C} be an $n$-th order track category and let
\begin{equation*}
	(X,\alpha)=\left( \cdots\xleftarrow{}X_{i-1}\xleftarrow{\alpha_i}X_i\xleftarrow{}\cdots,\ i\in\mathbb{Z} \right)
\end{equation*}
be a sequence of morphisms in the homotopy category $\cat{C}[\pt]$. A $\cat{C}$-\emph{chain complex $(X,f,F)$ associated
to $(X,\alpha)$} is defined in the same way as a representation of a higher Toda bracket. In fact, $(X,f,F)$ consists of
the following data. For $i\in\mathbb{Z}$ the element $f_i$ in $\cat{C}(\pt)$ is a representative of $\alpha_i$. Then
\begin{equation*}
	f^1_i\in\cat{C}(I),\quad f_i^1:f_i f_{i+1}\simeq 0
\end{equation*}
Assume now $f_i^k\in\cat{C}(I^k),\ 1\leq k<n$, is given. Then $f^{k+1}_i\in\cat{C}(I^{k+1})$ is a cubical extension of
$F^k_i\in\cat{C}(T^k)$ where $F_i^k$ is obtained by gluing as above, that is, $F^k_i$ has the restrictions
\begin{align*}
	&F_i^k|I^k\times 0= f_i^k \varepsilon_0(f_{i+k+1})\\
	&F_i^k|0\times I^k= \varepsilon_0(f_i)f_{i+1}^k\\
	&F_i^k|I^r\times 0\times I^k=f_i^r\otimes f_{i+r+1}^k,
\end{align*}
where $r+s=k,\ 1\leq r\leq k-1$.  Now $f^k_i$ and $F^k_i,\ 1\leq k\leq n$, describe $(X,f,F)$. This is a $\cat{C}$-chain
complex if the obstruction $\ob(F^n_i)=0$ vanishes for $i\in\mathbb{Z}$, see \ref{TrivBdDef}.

\begin{remark}
	For $n=1$ we have the track category $\cat{C}$ and in this case a $\cat{C}$-chain complex is the same as a
	$\emph{secondary chain complex}$ in \cite{BJSe}. Secondary chain complexes form a category but $\cat{C}$-chain
	complexes in general do not though morphisms between $\cat{C}$-chain complexes can be defined.
\end{remark}

\section{Higher order cohomology operations and the Adams spectral sequence}\label{HigOrdCoh}

Let $p$ be a prime and let $\mathbb{F}=\mathbb{Z}/p$ be the field of $p$ elements. Let $Z^n=K(\mathbb{F},n)$ be the
Eilenberg-MacLane space which is an $H$-group.

\begin{defn}
	The \emph{track theory of $n$-th order cohomology operations} is the abelian $n$-th order track category
	\begin{equation*}
		\cat{C}^n\{\calZ\}
\end{equation*}
where $\cat{C}$ is the topological $\mathbb{B}$-category and $\calZ$ is the set of all products
$Z^{n_1}\times\ldots\times Z^{n_r}$ with $n_1,\ldots,n_r\geq 1$ and $r\geq 1$. For a pointed space $X$ let
\begin{equation*}
	\cat{C}^n\{X\to\calZ\}
\end{equation*}
be the under category in section \ref{AbeDef} which is also an abelian $n$-th order track category. Here
$\cat{C}^n\{X\to\calZ\}$ is considered as a left module over $\cat{C}^n\{\calZ\}$.
\end{defn}

\begin{remark}
	The homotopy category $\cat{C}^n\{\calZ\}[\pt]$ is the theory of Eilenberg-MacLane spaces constructed in \cite[1.1.5]{BSe}.
	Models of this theory are connected unstable algebras over the Steenrod algebra $\calA$. For example 
	$\cat{C}\{X\to\calZ\}[\pt]$
	is such a model which is equivalently given by the cohomology $H^*(X,\mathbb{F})$. Given a sequence
	\begin{equation*}
		X\xrightarrow{\beta}X^0\xrightarrow{\alpha^0}X^1\xrightarrow{\alpha^1}\cdots\xrightarrow{\alpha^n}X^{n+1}
	\end{equation*}
	in $\cat{Top}^*\modu$ with $X^i\in\calZ$ the associated Toda bracket $\left<\alpha^n,\ldots,\alpha^0,\beta\right>$ is
	termed a \emph{higher matrix Massey product} in the $\calA$-module $H^*(X)$. If $X\in\calZ$ this is a higher Massey
	product in the Steenrod algebra.
\end{remark}

\begin{remark}
	It is possible to describe the analogue of the topological $\mathbb{B}$-de\-for\-ma\-tion category \cat{C} in
	$\cat{Top}^*$ in the stable homotopy category of spectra. For example the stable track theory of Eilenberg-MacLane
	spaces in \cite[2.2.6]{BSe} uses Eilenberg-MacLane spectra. The use of spectra, however, leads to technical
	complications which we want to avoid in this paper. We therefore use the ``stable range'' in the next definition.
\end{remark}

\begin{defn}
	For a ``large'' number $N$ let $\calZ_N$ be the subset of $\calZ$ above consisting of all products
	$Z^{n_1}\times\ldots\times Z^{n_r}$ with $N\leq n_i<2N,\ i=1,\ldots,r$. This is a \emph{stable range} of $\calZ$.
	Accordingly we get the \emph{stable theories}
	\begin{equation*}
		\cat{C}^n\{\calZ_N\},\quad\cat{C}^n\{\Sigma^NX\to \calZ_N\}
	\end{equation*}
	where $\Sigma^NX$ is the $N$-fold suspension of a $CW$-complex $X$ with $\dim(X)<N$.
\end{defn}

Let $X$ and $Y$ be finite $CW$-complexes. Then the \emph{Adams spectral sequence} $(E_n,n\geq 2)$ with
\begin{equation*}
	E_2=\calE xt_\calA(H^*X,H^*Y)
\end{equation*}
converges to the $p$-local part of the stable homotopy set $\{ Y,X \}$. For the computation of $\calE xt$ we
choose a resolution of the left $\calA$-modules $H^*$ by finitely generated free modules $M_i$
\begin{equation*}
	H^*X\leftarrow M_0\leftarrow M_1\leftarrow\cdots
\end{equation*}
Let $B_i$ be a basis of $M_i$ and let 
\begin{equation*}
	X^i=\bigtimes_{b\in B_i}\calZ^{N+|b|}
\end{equation*}
Then a finite part $H^*X\leftarrow M_0\leftarrow\cdots\leftarrow M_n$ of the resolution corresponds to a sequence
\begin{equation*}
	(X,\alpha)=\left(
	\Sigma^NX\xrightarrow{\alpha^{-1}}X^0\xrightarrow{\alpha^0}X^1\xrightarrow{\alpha^1}\cdots\xrightarrow{}X^m \right)
\end{equation*}
which for large $N$ lies in the homotopy category $\cat{K}[\pt]$ with $\cat{K}=\cat{C}^n\{W\to\calZ_N\},\ W=\Sigma^NX\vee
\Sigma^{N+s}Y$. Moreover an element $\overline{\beta}\in \mathrm{Ext}^r_\calA(H^*X,H^*Y)^s$ gets represented by a cocycle
$c\in\mathrm{Hom}(M_r,\Sigma^{-s}H^*Y)$ which corresponds to a map $\beta$ in the following diagram which for $N$ large lies in
$\cat{K}[\pt]$.
\begin{equation*}
	\xymatrix{
	\Sigma^NX\ar[r]	&	X^0\ar[r]^{\alpha^0}& \cdots
	\ar[r]&X^r\ar[r]^-{\alpha^r}&\cdots\ar[r]^-{\alpha^{r+n}}&X^{r+n+1},\ r+n+1\leq m,\\
	&&&\Sigma^{N+s}Y\ar[u]_{\beta}&
	}
\end{equation*}
We use this diagram for the determination of the differential
\begin{equation*}
	d_{n+1}:E_n^{r,s}\to E_n^{r+n+1,s+n}
\end{equation*}
\begin{prop}
	There is a \cat{K}-chain complex $(X,f,F)$ associated to $(X,\alpha)$ above. Moreover since $\overline{\beta}$
	represents an element in $E^{r,s}_n$ there is a $\cat{K}^{n-1}$-chain complex $(Y,g,G)$ associated to
	\begin{equation*}
		(Y,\beta)=\left( \Sigma^{n+s}Y\xrightarrow{\beta}X^r\xrightarrow{\alpha^r}\cdots\xrightarrow{}X^{r+n+1} \right).
	\end{equation*}
	Here the restriction of $(Y,g,G)$ to $(X^r\to\cdots\to X^{r+n+1}$ is given by $(X,f,F)$ and $\cat{K}^{n-1}$ is the
	truncation of \cat{K}. Since $(Y,\beta)$ is a $\cat{K}^{n-1}$-chain map we can choose a cubical extension defining
	$G^n_0\in\cat{K}(T^n)$. The obstruction $\ob(G_0^n)\in [ \Sigma^n(\Sigma^{N+s}Y), X^{r+n+1} ]$ yields an
	element in $\calE xt_\calA^{r+n+1}(H^*X,H^*Y)^{n+s}$ which represents $d_{n+1}\{ \beta \}\in
	E_n^{r+n-1,n+s}$.
\end{prop}

This result is proved for $n=1$ in \cite{BJ}.

\begin{cor}
	Let $\cat{K}'$ be an $n$-th order track category quasi isomorphic to $\cat{K}=\cat{C}^n\{W\to \calZ_N\},\ W=\Sigma^NX\vee
	\Sigma^{N+s}Y$. Then the differential $d_n$ can be computed in $\cat{K}'$.
\end{cor}

For $n=1$ this is proved in \cite[5.1]{BJSe}.

\section{$\Delta$-balls}

The diagonal $\Delta$ of a $CW$-complex $X$ is not a cellular map but is homotopic to a cellular map $\overline{\Delta}$
which is called a \emph{diagonal approximation}. Then $\overline{\Delta}=\overline{\Delta}_X$ induces a chain map
\begin{equation*}
	\overline{\Delta}_*:C_*X\to C_*(X\times X)=C_*X\otimes C_*X
\end{equation*}
where $C_*X$ is the cellular chain complex with $C_nX=H_n(X^n,X^{n-1})$. We say that
$(C_*X,\overline{\Delta}_*)$ is a \emph{coalgebra} if $\overline{\Delta}_*$ is coassociative and the augmentation
$\varepsilon_0=(\varepsilon_0')_*:C_*B\to C_*(\pt)=R$ satisfies $(1\otimes
\varepsilon_0)\overline{\Delta}_*=1=(\varepsilon_0\otimes 1)\overline{\Delta}_*$. In general it is not possible to find
a diagonal approximation $\overline{\Delta}$ such that $(C_*X,\overline{\Delta}_*)$ is a coalgebra, but we consider
$CW$-complexes with the following nice properties.

\begin{defn}
	A $\Delta$-\emph{$CW$-complex} $X$ is a regular $CW$-complex together with a diagonal approximation
	$\overline{\Delta}$ and a homotopy $D:\overline{\Delta}\simeq \Delta$ in \cat{Top} such that the following properties
	are satisfied.
	\begin{enumerate}[label=(\alph*)]
		\item Each subcomplex $Y$ of $X$ admits a commutative diagram
			\begin{equation*}
				\xymatrix{
				X\ar[r]^{\overline{\Delta}}& X\times X\\
				Y\ar[u]\ar[r]^{\overline{\Delta}_Y}&Y\times Y\ar[u]
				}
			\end{equation*}
			and the homotopy $D$ induces a homotopy $D_Y:\overline{\Delta}_Y\simeq\Delta_Y$. 
		\item The cellular chain complex $(C_*X,\overline{\Delta}_*)$ is a coalgebra.
	\end{enumerate}
	We have the following properties of $\Delta$-$CW$-complexes.
	\begin{enumerate}[label=(\alph*)]
			\setcounter{enumi}{2}
		\item The interval $I$ is a $\Delta$-$CW$-complex using $\overline{\Delta}$ in $I\times 0\cup 1\times I$.
		\item The product $X\times Y$ of $\Delta$-$CW$-complexes is a $\Delta$-$CW$-complex given by
			\begin{equation*}
				\overline{\Delta}=(1\times T\times 1)(\overline{\Delta}_X\times \overline{\Delta}_Y),\quad D^t=(1 \times T\times
				1)(D^t_X\times D^t_Y),\ t\in I,
			\end{equation*}
			where $T:X\times Y\to Y\times X$ is the interchange map.
		\item A subcomplex $Y$ of a $\Delta$-$CW$-complex $X$ is a $\Delta$-$CW$-complex. 
		\item Let $X,X'$ be $\Delta$-$CW$-complexes and let $Y\subset X$ and $Y\subset X'$ be the inclusions of
			$\Delta$-$CW$-subcomplexes. Then the union $X\cup_Y X'$ is a $\Delta$-$CW$-complex with
			$\overline{\Delta}=\overline{\Delta}_X\cup \overline{\Delta}_{X'}$.
		\item Let $X,X',Y$ be $\Delta$-$CW$-complexes and let $Y\subset X$ be the inclusion of a $\Delta$-$CW$-subcomplex.
			Then the pushout $P$ as in the diagram
			\begin{equation*}
				\xymatrix{
				X\times X'\ar[r]&P\\
				Y\times X'\ar[u]\ar[r]_q&X'\ar[u]
				}
			\end{equation*}
			is a $\Delta$-$CW$-complex. Here $q$ is the projection. For example a relative cylinder is such a pushout.
	\end{enumerate}
\end{defn}

\begin{defn}
	A $\Delta$-ball is a $\Delta$-$CW$-complex for which the underlying $CW$-complex is a ball, see section \ref{IndSet}.
\end{defn}

\begin{prop}
	Each ball $B$ in the indexing set $\mathbb{B}$ of balls in section \ref{IndSet} has the canonical structure of a
	$\Delta$-ball.
\end{prop}

This follows from the properties of $\Delta$-$CW$-complexes above, compare the definition of $\mathbb{B}$ in section
\ref{IndSet}.

\section{Track categories associated to truncated chain algebras}\label{TruChaAlg}

Let $R$ be a commutative ring with unit. We use the category of $R$-modules with tensor product $\otimes=\otimes_R$. 

A \emph{chain algebra} $Q$ is a non-negatively bigraded $R$-module $Q=\{ Q^r_s; r,s\geq 0 \}$ with unit $1\in
Q_0^0$, associative multiplication
\begin{equation*}
	\mu:Q^r_s\otimes Q^{r'}_{s'}\to Q^{r+r'}_{s+s'},\ \mu(x\otimes y)=x\cdot y, 
\end{equation*}
and differential $d:Q_s^r\to Q^r_{s-1}$ satisfying $d\circ d=0$ and 
\begin{equation*}
	d(x\cdot y)=(dx)\cdot y+(-1)^s x\cdot(dy).
\end{equation*}
(If $Q$ is concentrated in upper degree 0 then $Q$ is a chain algebra in the usual sense.) A $Q$-\emph{module} is a non
negatively bigraded $R$-module $M=\{ M_s^r;r,s\geq 0 \}$ with a differential $d:M^r_s\to M^r_{s-1},\ d\circ
d=0$, and an action
\begin{equation*}
	\mu:M^r_s\otimes Q^{r'}_{s'}\to M^{r+r'}_{s+s'},\ \mu(x\otimes y)=x\cdot y,
\end{equation*}
satisfying the formula for $d(x\cdot y)$ above. (If $M$ is concentrated in upper degree $0$ then $M$ is a chain complex
in the usual sense.) A $Q$-\emph{morphism} $f:M\to N$ between $Q$-modules is a map $f:M^r_s\to N^r_s$ with $df=fd$ and
$f(x\cdot y)=f(x)\cdot y$.

Below we shall consider $Q$-modules of the form
\begin{equation*}
	C\otimes L\otimes Q
\end{equation*}
where $C$ is a chain complex and $L$ is a finitely generated free graded $R$-module. Here $C$ and $L$ are concentrated
in upper degree 0. 

\begin{defn}
	A chain algebra $Q$ is $n$-\emph{truncated} if $Q^r_s=0$ for $s>n$. The $n$-\emph{th truncation} $Q(n)$ of a chain
	algebra $Q$ is given by
	\begin{equation*}
		Q(n)^r_s=
		\begin{cases}
			0&\text{for }s>n\\
			Q^r_n/dQ^r_{n+1}&\text{for }s=n\\
			Q^r_s&\text{for }s<n
		\end{cases}
	\end{equation*}
	Then $Q(n)$ is an $n$-truncated chain algebra.
\end{defn}

We now define for an $n$-truncated chain algebra $Q$ the $n$-\emph{th order track category $\cat{K}^Q$ associated to
$Q$} by the properties in (1), (2) and (3) below.

\begin{enumerate}
	\item The objects of $\cat{K}^Q$ are the finitely generated free graded $R$-modules; they form the set $\calO$
		of objects. The zero object $0\in\calO$ is the trivial module. For $B\in \mathbb{B}(n)$ let $\cat{K}^Q(B)$ be
		the category with objects in $\calO$ and morphisms $f:L\to L'$ with $L,L'\in\calO$ given by $Q$-morphisms
		\begin{equation*}
			C_*(B)\otimes L\otimes Q\xrightarrow{f}L'\otimes Q
		\end{equation*}
		Composition of such morphisms $gf:L\to L'\to L''$ is defined by the composite
		\begin{equation*}
			C_*B\otimes L\otimes Q\xrightarrow{\overline{\Delta}_*}C_*B\otimes C_*B\otimes L\otimes Q\xrightarrow{C_*B\otimes
			f}C_*B\otimes L'\otimes Q\xrightarrow{g}L''\otimes Q
		\end{equation*}
		Since $(C_*B,\overline{\Delta}_*)$ is a coalgebra this is a well defined category. The identity $1:L\to L$ is given
		by $C_*(B)\otimes L\to L$ induced by $(\varepsilon_0')_*:C_*(B)\to C_*(\pt)=R$.
		Let $j:A\to B$ be a morphism in $\mathbb{B}(n)$ then $j$ induces a coalgebra morphism $j_*:C_*A\to C_*B$ and the
		functor $j^*:\cat{K}^Q(B)\to \cat{K}^Q(A)$ carries $f$ to $f\circ (j_*\otimes L\otimes Q)$.
\end{enumerate}

\begin{lem}
	The relation of $\simeq$ defined by $\cat{K}^Q$ is a natural equivalence relation.
\end{lem}

\begin{proof}
	For the ball $JB$ we know that $C_*(JB)$, as a quotient of $C_*I\otimes C_*B$, is a relative cylinder in the category
	of chain complexes.
\end{proof}

Hence the category $\cat{K}^Q[B]=\cat{K}^Q(B)\modu$ is well defined.

\begin{lem}
	If $\dim(A)=n$ then $\simeq$ on $\cat{K}(B)$ is the trivial relation so that $\cat{K}(B)=\cat{K}[B]$.
\end{lem}

This follows readily by the assumption that $Q$ is $n$-truncated.

\begin{enumerate}
		\setcounter{enumi}{1}
	\item Let $j:A\to B$ be a map in $\mathbb{B}_{\partial}(n)$. Then $\partial j:\partial A\to \partial B$ induces a coalgebra
		map $j^{\partial}=C_*(\partial j)$ in the commutative diagram
		\begin{equation*}
			\xymatrix{
			C_*\partial A\ar@{}[d]|{\bigcap}\ar[r]^{j^{\partial}}	& C_*\partial B\ar@{}[d]|{\bigcap}\\
			C_* A\ar[r]^{\overline{j}}&C_*B
			}
		\end{equation*}
		Here we can choose a chain map $\overline{j}$ extending $j^{\partial}$ since $B$ is contractible. We now define
		\begin{equation*}
			j^*:\cat{K}[B]\to K[A]
		\end{equation*}
		by $j^*\{ f \}=\{ f(\overline{j}\otimes L\otimes Q) \}$
\end{enumerate}

\begin{lem}
	$j^*$ is a functor which does not depend on the choice of $\overline{j}$.
\end{lem}

\begin{proof}
	$j^{\partial}$ is a coalgebra map but $\overline{j}$ is only a chain map extending $j^{\partial}$. This yields the
	diagram
	\begin{equation*}
		\xymatrix{
		C_*A\ar[r]^{\overline{j}}\ar[d]_{\overline{\Delta}_*}	&	C_*B\ar[d]^{\overline{\Delta}_*}\\
		C_*A\otimes C_*A\ar[r]^-{\overline{j}\otimes \overline{j}} & C_*B\otimes C_*B
		}
	\end{equation*}
	which commutes on the boundary $C_*\partial A$. Since $B$ is contractible there is a homotopy
	$(\overline{j}\otimes\overline{j})\overline{\Delta}_*\simeq \overline{\Delta}_*\overline{j}\rel  C_*\partial
	A$. This shows that $j^*(gf)=(j^*g)(j^*f)$.
\end{proof}

\begin{enumerate}
		\setcounter{enumi}{2}
	\item We now consider the \emph{gluing functor} in $\cat{K}^Q$. Given ball pairs $(B,A)$ and $(B,A')$ in
		$\mathbb{B}(n)$ we have 
		\begin{equation*}
			C_*(B\cup_A B')=C_*B\cup_{C_*A}C_*B'
		\end{equation*}
		where the right hand side is a pushout of chain complexes. Given $F$ in $\cat{K}^Q(B)$ and $G$ in $\cat{K}^Q(B')$
		with $\delta_A F=\delta_A G$ we get $F\cup G$ by
		\begin{equation*}
			C_*(B\cup_A B')\otimes L\otimes Q=C_*B\otimes L\otimes Q\cup_{C_*A\otimes L\otimes Q}C_*B'\otimes L\otimes
			Q\xrightarrow{F\cup G}L'\otimes Q
		\end{equation*}
	\item Finally we obtain $\otimes$-\emph{products} in $\cat{K}^Q$ as follows. Let $g:Y\to Z\in\cat{K}^Q(B')$ and
		$f:X\to Y\in \cat{K}^Q(B)$. Then $g\otimes f:X\to Z\in\cat{K}^Q(B'\times B)$ is the composite
		\begin{equation*}
			C_*(B'\times B)\otimes X\otimes Q=C_*B'\otimes C_*B\otimes X\otimes Q\xrightarrow{C_*B'\otimes f}C_*B'\otimes
			Y\otimes Q\xrightarrow{g}Z\otimes Q.
		\end{equation*}
\end{enumerate}

\begin{thm}
	The data in $(1)\ldots(4)$ above describe a well-defined $n$-th order track category $\cat{K}^Q$ with all the
	properties in section \ref{HigOrdTra}. Moreover for the $(n-1)$-truncation we get
	\begin{equation*}
		(\cat{K}^Q)^{n-1}=\cat{K}^{Q(n-1)}.
	\end{equation*}
\end{thm}

The homology $H_0=H_0^*(Q)$ is a graded algebra. Let $\cat{mod}(H_0)$ be the category of finitely generated free right
$H_0$-modules $L\otimes H_0,\ L\in\cat{Ab}$. Then we get the bifunctor
\begin{align*}
	&D^k:\cat{mod}(H_0)^{\op}\times\cat{mod}(H_0)\to \cat{Ab}\\
	&D^k(L,L')=\mathrm{Hom}_{H_0}\left(L\otimes H_0,L'\otimes H_0\otimes_{H_0}H_k\right)
\end{align*}
where $H_k=H_k^*(Q)$ is an $H_0$-bimodule.

\begin{thm}
	The $n$-th order track category $\cat{K}^Q$ is abelian with homotopy category
	\begin{equation*}
		K^Q[\pt]=\cat{mod}(H^*_0Q)
	\end{equation*}
	and natural systems $D^k$ defined above $1\leq k\leq n$.
\end{thm}

\begin{remark}
	Higher order Toda brackets in the abelian $n$-th order track category $\cat{K}^Q$ coincide with \emph{higher order
	matrix Massey products} in the differential algebra $Q$.
\end{remark}

Let $\calM^N$ be the set of finitely generated free graded $R$-modules concentrated in degree $<N$. Then $\cat{K}^Q$
defines the $n$-th order track category $( \cat{K}^Q(\calM^N) )^{\op}$ with objects in $\calM^N$ which is
formally dual to $\cat{K}^Q(\calM^N)$.

\begin{conj}
	There exists a bigraded differential algebra $Q$ over $R=\mathbb{Z}/p^2$ such that for $n\geq 0$ the truncation $Q(n)$
	of $Q$ yields an $n$-th order track category $\left( \cat{K}^{Q(n)}(\calM^N) \right)^{\op}$ which is weakly equivalent
	to the stable track category $\cat{C}^n\{\calZ_N\}$ of higher cohomology operations. We call $Q$ the ``\emph{algebra of
	higher cohomology operations}''.
\end{conj}

\begin{thm}
	The conjecture is true for $n=0$ and $n=1$. For $n=0$ we get the Steenrod algebra $Q(0)=\calA$. For $n=1$ we get the
	pair algebra $Q(1)=\calB$ of secondary cohomology operations. The weak equivalence of track categories is established
	for $n=1$ in \cite[5.5.6]{BSe}.
\end{thm}

If the algebra $Q$ is computed one has by the conjecture and section \ref{HigOrdCoh} a direct way to compute the
differentials in the Adams spectral sequence which then allows the computation of stable homotopy groups of spheres.

\end{document}